\documentclass[12pt,a4paper]{article}
\usepackage[utf8]{inputenc}
\usepackage[T1]{fontenc}

\usepackage{fancyhdr}
\usepackage{lipsum}

\usepackage{amsthm}
\newtheorem{theorem}{Theorem}
\newtheorem{corollary}[theorem]{Corollary}

\theoremstyle{definition}

\usepackage{amsfonts}
\usepackage[figuresright]{rotating}
\usepackage{amssymb}
\usepackage{amsmath}
\usepackage{graphics}
\usepackage{enumerate}
\usepackage{tikz}
\usepackage{caption}
\usepackage{subcaption}
\usepackage{float}
\usepackage{comment}
\DeclareMathOperator{\QEC}{QEC}

\newcommand{\RE}{\mbox{\rm Re\,}}

\newcommand{\norm}[1]{\left\| #1\right\|}

\title{Quadratic Embedding of Theta Graphs via Reproducing Kernel Hilbert Spaces}
\author{Marek Skrzypczyk\thanks{Faculty of Mathematics and Computer Science, Jagiellonian University, {\L}ojasiewicza 6, 30-348 Krak\'{o}w, Poland. E-mail: {\tt marekskrzypczyk1@o2.pl}.}}

\begin{document}
\maketitle
\date{}

\begin{abstract}
The quadratic embedding property of graphs consisting of three paths (theta graphs) is fully characterised. For this aim,  a theorem by Winkler (1985)  is utilized. An alternative proof of that result using the RKHS technique is presented.
\end{abstract}


\medskip\medskip

\emph{Dedicated to the memory of Professor Franciszek Hugon Szafraniec}

\section{Introduction}

Let $G=(V,E)$ be a connected graph (finite and infinite). Denote by $d(x,y)$ the graph distance between two vertices $x,y\in V$, i.e., the length of a shortest walk connecting $x$ and $y$. A \textit{quadratic embedding} of a connected graph $G=(V,E)$ in a Hilbert space $\mathcal{H}$ is a map $\phi: V \to \mathcal{H}$ satisfying
$$||\phi(x) - \phi(y)||^2=d(x,y), \quad x,y \in V, $$
where $||\cdot||$ is the norm of the Hilbert space $\mathcal{H}$. A graph $G$ is said to be of \emph{QE class} if it admits a quadratic embedding, otherwise it is said to be of \emph{non-QE class}. Schoenberg showed that the graph is of QE class if and only if its distance matrix $D=[d(x,y)]_{x,y \in V}$ is conditionally negative definite \cite{Sch2}. For bipartite graphs, a quadratic embedding is equivalent to an isometric embedding into a hypercube \cite{Djokovic,RothWinkler,Tan}. 

The concept of quadratic embedding can be applied not only to graphs but also to metric spaces in general. It has been studied largely in conjunction with Euclidean distance geometry \cite{Alfakih,Bal,Bapat,Jak1,Jak2,Lib,Maehara}. Moreover, a quadratic embedding has also appeared in harmonic analysis, for example, in the study of the Kazhdan property $(T)$ of groups \cite{Boz2, BozDolEjsGal, BozJanSpa, Haagerup, Obata5, Obata6}, as well as in quantum probability \cite{Boz2, BozJanSpa, Hora1, Obata5, Obata6}.

Consider two graphs $H=(V',E')$ and  $G=(V,E)$. Then $H$ is called a subgraph of $G$ if $V'\subseteq V $ and $E'\subseteq E$. Assume that both graphs $H$ and $G$ are connected, with graph distances denoted respectively by $d_H$ and $d_G$.  We say that $H$ is \emph{isometrically embedded} in $G$ if for each pair of vertices of $V'$ holds 
	$$d_H (x,y)= d_G (x,y).$$
Assume that $H=(V',E')$ is isometrically embedded subgraph of $G=(V,E)$. It is easy to observe: if $H$ is of non-QE class, then $G$ is of non-QE class as well. Hence, Obata divided non-QE graphs into two subclasses. A graph of non-QE class is called \emph{primary} if it contains no isometrically embedded proper subgraphs of non-QE class \cite{Obata1}. Otherwise, it is called \textit{non-primary}. All primary non-QE graphs on at most 6 vertices are known \cite{ObataPrim6v}. Furthermore, Obata in \cite{Obata1} showed all primary non-QE complete multipartite graphs. 

The natural problem is whether primary non-QE graphs exist for each number of vertices greater than or equal to 5. In \cite{MSWtheta}, Młotkowski, Skrzypczyk, and Wojtylak provided a positive answer to this question, defining \emph{theta graphs} $\Theta(\alpha,\beta,\gamma)$ as follows. Let $P_{n+1}=[x_0,\dots,x_n]$ denote the path graph, i.e., a graph with $n+1$ vertices $\{x_0,\ldots,x_n\}$ and $n$ edges $\{x_{i-1},x_i\}$ ($i=1,\ldots,n$). For given integers $1\leq \alpha,\beta,\gamma$, with at most one of them equal $1$,
we define a \textit{theta graph} $\Theta(\alpha,\beta,\gamma)$ by taking three path graphs 
$$
P_{\alpha+1}=  [x_0,\ldots,x_{\alpha}] ,\ 
P_{\beta+1}=   [y_0,\ldots,y_\beta],\ 
P_{\gamma+1}= [z_0,\ldots,z_\gamma],\ 
$$
with the same endpoints $x_0=y_0=z_0$,  $x_{\alpha}=y_{\beta}=z_{\gamma}$ and the remaining vertices $x_j$ ($j=1,\dots,\alpha-1$), $y_j$  ($j=1,\dots,\beta-1$) and $z_{j}$ ($j=1,\dots,\gamma-1$) being mutually different.
A theta graph $\Theta(\alpha,\beta,\gamma)$ has $\alpha+\beta+\gamma-1$ vertices, $\alpha+\beta+\gamma-1$ edges, and is bipartite if and only if $\alpha,\beta,\gamma$ are of the same parity. The graph $\Theta(2,3,5)$ is presented in Figure \ref{ex235}.

\begin{figure}[H]
\centering
\begin{tikzpicture}[x=12cm,y=12cm]
\tikzset{     
    e4c node/.style={circle,fill=black,minimum size=0.25cm,inner sep=0}, 
    e4c edge/.style={line width=1.5pt}
  }
  \node[e4c node] (1) at (0.15, 0.85) {};
   \node[anchor= south] at (0.15, 0.86) {$x_2=y_{3}=z_{5}$};
  \node[e4c node] (2) at (0.05, 0.70) {};
   \node[anchor= east] at (0.05, 0.70) {$x_1$};
  \node[e4c node] (3) at (0.15, 0.55) {};
  \node[anchor= north] at (0.15, 0.54) {$x_0=y_{0}=z_{0}$};
  \node[e4c node] (4) at (0.25, 0.77) {}; 
  \node[anchor= north east] at (0.25, 0.77) {$y_2$};
  \node[e4c node] (5) at (0.25, 0.63) {};
  \node[anchor= south east] at (0.25, 0.63) {$y_1$};
  \node[e4c node] (6) at (0.30, 0.80) {};
  \node[anchor= south west] at (0.30, 0.80) {$z_4$};
  \node[e4c node] (7) at (0.45, 0.75) {};
  \node[anchor= south west] at (0.45, 0.75) {$z_3$};
  \node[e4c node] (8) at (0.45, 0.65) {};
  \node[anchor= north west] at (0.45, 0.65) {$z_2$};
 \node[e4c node] (9) at (0.30, 0.60) {};
 \node[anchor= north west] at (0.30, 0.60) {$z_1$};

  \path[draw,thick]
  (1) edge[e4c edge]  (2)
  (1) edge[e4c edge]  (4)
  (2) edge[e4c edge]  (3)
  (4) edge[e4c edge]  (5)
  (5) edge[e4c edge]  (3)
  (1) edge[e4c edge]  (6)
  (6) edge[e4c edge]  (7)
  (7) edge[e4c edge]  (8)
  (8) edge[e4c edge]  (9)
  (9) edge[e4c edge]  (3)
  ;

\end{tikzpicture}

\caption{The graph $\Theta(2,3,5)$} \label{ex235}
\end{figure}
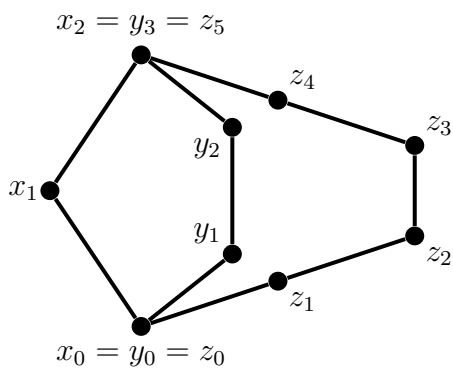

Assume that $1\le\alpha\le\beta\le\gamma$, $\beta\ge2$. It was shown in \cite{MSWtheta} that $\Theta(1,\beta,\gamma)$ is of QE class if either $\beta=2$, or $\beta=3$, or $\beta$ and $\gamma$ are odd. As well as $\Theta(\alpha,\beta,\gamma)$ is of non-QE class if either $\alpha=\beta=2$, or $\alpha=2$, $\beta=3$, $\gamma$ is even, or $\alpha=2$, $\beta\ge4$, or $\alpha\ge3$. Moreover, each theta graph of non-QE class is also primary, hence these results imply that for each number of vertices greater than or equal to 5, there exists a primary non-QE graph.

To complete the characterization of theta graphs $\Theta(\alpha,\beta,\gamma)$ according to the QE property, we need to verify the cases when
\begin{itemize}
\item $\alpha=1$, $4\le\beta\le\gamma$ and at least one of $\beta$, $\gamma$ is even, and
\item $\alpha=2$, $3=\beta\le\gamma$ and $\gamma$ is odd.
\end{itemize}

In this paper, we show that every theta graph $\Theta(1,\beta,\gamma)$ is of QE class, as well as theta graphs $\Theta(2,3,\gamma)$ are of QE class for $\gamma\in \{3,5,7\}$. Thus, we obtain the characterization of theta graphs $\Theta(\alpha,\beta,\gamma)$ according to the QE property.

\begin{theorem}\label{mainthm}
Assume that $1\le\alpha\le\beta\le\gamma$, $\beta\ge2$.
Then $\Theta(\alpha,\beta,\gamma)$ is of QE class if and only if either
\begin{itemize}
\item $\alpha=1$, or
\item $\alpha=2$, $\beta=3$ and $\gamma\in \{3,5,7\}$.
\end{itemize}
\end{theorem}

The methods used in \cite{MSWtheta} are not suitable to complete this characterization. Hence, we utilize the technique proposed by Winkler \cite{Winklertree}. Let $G$ be a graph. Consider an orientation $G'$ of $G$, i.e., for each edge $(u,v)$ of $G$ exactly one of the directed edges $(a,b)$ or $(b,a)$ is an edge of $G'$.
\begin{theorem}[Winkler \cite{Winklertree}\label{tree}]
Let $G$ be a graph on $n+1$ vertices and let $G'$ be an arbitrary orientation of $G$. Let $T$ be an arbitrary spanning tree of $G$, and let $e_1=(a_1,b_1),e_2=(a_2,b_2),\ldots,e_n=(a_n,b_n)$ denote the edges of $T$, oriented as in $G'$. Then the following conditions are equivalent:
\begin{enumerate}[\rm (i)]
\item $G$ is of QE class,
\item the kernel $K(e_i,e_j)=\frac{1}{2}\big(d_G(a_i,b_j)-d_G(a_i,a_j)-d_G(b_i,b_j)+d_G(b_i,a_j)\big)$ is positive semidefnite.
\end{enumerate}
\end{theorem}
Below, we present a new proof of Winkler's theorem. For this aim, we use the RKHS technique,  one of the favorite topics of Franciszek Szafraniec, cf., e.g, \cite{Alpay,Szafraniec4,Szafraniec3}. Note that the original proof in \cite{Winklertree} consists of  a rather technical combinatorial transformation of Schoenberg's theorem.

This paper is arranged as follows. Section \ref{prel} contains the necessary background. In Section~\ref{sec.winkler} we present the alternative version of the proof of Theorem~\ref{tree}.
In Section~\ref{sec.alpha1}, we study the quadratic embedding of $\Theta(1,\beta,\gamma)$, utilizing Winkler's theorem. We conclude this section by analyzing QE constants of these graphs. Finally, in Section \ref{sec.alpha2}, we explore the quadratic embedding of $\Theta(2,3,\gamma)$ for odd $\gamma\ge3$. To show that theta graphs $\Theta(2,3,3)$, $\Theta(2,3,5)$, and $\Theta(2,3,7)$ are of QE class, Winkler's theorem is applied. Whereas, the non-embeddability of graphs $\Theta(2,3,\gamma)$ for odd $\gamma\ge 9$ is proven by explicit calculation.

\section{Preliminaries}\label{prel}
A \textit{graph} $G=(V,E)$ is a pair of a non-empty set $V$, whose elements are called \emph{vertices}, and a set $E$, which is a set of unordered pairs of two distinct vertices in $V$, whose elements are called \emph{edges}. A \textit{bipartite graph} is a graph whose vertex set can be partitioned into two subsets $V_1$ and $V_2$, such that no two vertices within the same set are adjacent. A \textit{walk} is a finite sequence of edges that joins a sequence of vertices. A \emph{trail} is a walk in which all edges are distinct. If the vertices of a trail are distinct, then it is called a \emph{path}. If only the first and last vertices of the trail are the same, then it is called a \emph{cycle}. A \emph{tree} is a graph without a cycle. A \emph{spanning tree} of a graph $G$ is a subgraph that is a tree which includes all of the vertices of $G$. An \emph{orientation} of a graph $G$ is an assignment of exactly one direction to each of the edges of $G$, i.e., for each edge $(u,v)$ we have exactly one of the edges $(a,b)$ or $(b,a)$. A graph with an orientation is called an \emph{oriented graph}.

A graph is called \textit{connected} if there exists a finite path for each pair of vertices. 
In such a case, by $d(x,y)=d_G(x,y)$ we denote the length of a shortest path connecting $x$ and $y$, such a path is called \textit{geodesic}. Note that  $d(x,y)$ is a metric on $V$, called the \textit{graph distance}. The \textit{distance matrix} is defined by
$$
D=d(x,y)_{x,y \in V}.
$$ 

An $n\times n$ real symmetric matrix $M$ is called \emph{positive semidefinite} if $\langle f,Mf\rangle \ge 0$ for all $f\in \mathbb{R}^n$. Similarly, an $n\times n$ real symmetric distance matrix $D$ is called \emph{conditionally negative definite} if $\langle f,Df\rangle \le 0$ for all $f\in \mathbb{R}^n$ with $\langle\mathbf{1},f\rangle=0$, where $\mathbf{1}$ denotes the all-ones vector in~$\mathbb{R}^n$.

For the definitions of quadratic embedding, isometrically embedded subgraphs, and primary non-QE graphs, we refer to the Introduction. In recent years, Obata \cite{Obata3} developed the theory of quadratic embedding by introducing the  \emph{QE constant} of a graph $G$, which is defined~as 
	$$
	\QEC(G)= \max \{  \langle f,Df\rangle :  \norm{f} =1, \langle\mathbf{1},f\rangle=0 \},
	$$
where $\textbf{1}$ is a vector of ones. A graph $G$ is of QE-class if and only if   
$\QEC(G) \le 0$. Explicit values of QEC for some classes of graphs can be found in
\cite{Bas, ObataBaskoro2024, ChouNan, IraSug, Jak1, Jak2, Lou, Mlo1, Mlo2, MłotkowskiObata2024, MSWtheta, Obata1, ObataPrim6v, Obata3, Pur}. Clearly, if $H$ is isometrically embedded in $G$, then $\QEC(H)\le \QEC(G)$. 

The definition of theta graphs is given in the Introduction.

\section{Winkler's Theorem}\label{sec.winkler}
Winkler proposed in \cite{Winklertree} a method to study the quadratic embedding of graphs. Let $G$ be a graph. Consider an orientation $G'$ of $G$, i.e., for each edge $(u,v)$ of $G$, exactly one of the edges $(a,b)$ or $(b,a)$ is an edge of $G'$. 
We refer, e.g., to \cite{Szafraniec2,Szafraniec1} for the basics on RKHS.

\begin{proof}[Proof of Theorem \ref{tree}]
   (i)$\Rightarrow$(ii) Consider the graph $G$ on $n+1$ vertices. If $G$ is of QE class, then
$$
d_G(x,y)=\norm{ \phi(x)-\phi(y) }^2=\norm{\phi(x)}^2+\norm{\phi(y)}^2-2\RE\langle  \phi(x),\phi(y)\rangle,
$$
for each pair of vertices.
Since $d_G(x,y)$ is real, $\langle  \phi(x),\phi(y)\rangle$ is real as well, and thus 
\begin{equation}\label{inner}
\langle  \phi(x),\phi(y)\rangle =\frac{1}{2}\Big(\norm{\phi(x)}^2+\norm{\phi(y)}^2 -d_G(x,y)\Big).
\end{equation}

Let $G'$ be an arbitrary orientation of $G$. Let $T$ be an arbitrary spanning tree of $G$, and let $e_1=(a_1,b_1),e_2=(a_2,b_2),\ldots,e_n=(a_n,b_n)$ denote the edges of $T$, oriented as in $G'$. For the edges from $T$, define 
$$K_{e_i}:=\phi(b_i)-\phi(a_i). $$

Clearly,
\begin{equation}\label{inner3}
    \langle K_{e_i},K_{e_j}\rangle=\langle \phi(b_i),\phi(b_j)\rangle-\langle \phi(b_i),\phi(a_j)\rangle-\langle \phi(a_i),\phi(b_j)\rangle+\langle \phi(a_i),\phi(a_j)\rangle.
\end{equation}
Applying (\ref{inner}) to (\ref{inner3}), we obtain
\begin{equation}\label{inner4}\langle K_{e_i},K_{e_j}\rangle :=\frac{1}{2}\big(d_G(a_i,b_j)-d_G(a_i,a_j)-d_G(b_i,b_j)+d_G(b_i,a_j)\big). \end{equation}
Consequently, there exists a positive semidefinite kernel $K:X\times X\to \mathbb{C}$, where $X$ is a set of edges of $T$, such that $K(e_i,e_j)=\frac{1}{2}\big(d_G(a_i,b_j)-d_G(a_i,a_j)-d_G(b_i,b_j)+d_G(b_i,a_j)\big)$.

(i)$\Leftarrow$(ii) 
Let $\mathcal H$ be the RHKS corresponding to the kernel and let $K_e\in\mathcal H$ ($e\in T$) be the kernel elements.
Take the spanning tree and define the mapping $\phi:V\to \mathcal H$ recursively. Fix one vertex $v_0$ as a root of $T$, set $\phi(v_0)=0$. 
Next, we extend the definition along the spanning tree.
Assume that $v$ is a vertex for which the value of $\phi$ is not defined yet, and let $e=(v,w)$ be an edge in $G$ with $\phi(w)$ already defined.
Define $\phi(v)=\phi(w)\pm K_e$ depending on the orientation, so that 
\eqref{inner4} holds. 

Let $P_T(x,y)$ be the path in $T$ which connects vertices $x$ and $y$, and let $d_T(x,y)$ be the number of edges in $P_T(x,y)$. For the edges $e_i=(a_i,b_i)\in P_T(x,y)$, define the function
$$\zeta(e_i)=\begin{cases}
1 & \text{if}\ d_T(x,a_i)<d_T(x,b_i)\\
-1 & \text{if}\ d_T(x,a_i)>d_T(x,b_i)
\end{cases}.$$
Observe that $$\phi(y)-\phi(x)=\sum_{e_i\in P_T(x,y)} \zeta(e_i)K_{e_i}.$$
Then

\begin{equation*}
\begin{split}
\|\phi(y)-\phi(x)\|^2
&=\|\sum_{e_i\in P_T(x,y)} \zeta(e_i)K_{e_i} \|^2 \\ &=\sum_{e_i\in P_T(x,y)}\|K_{e_i}\|^2
 +2 \sum_{\substack{
e_i,e_j\in P_T(x,y)\\
e_i\ne e_j
}}
\zeta(e_i)\zeta(e_j)\langle K_{e_i},K_{e_j}\rangle.
\end{split}
\end{equation*}
Note that $\|K_{e_i}\|^2=1$ for each $e_i\in P_T(x,y)$. By simple calculations, we obtain
$$2 \sum_{\substack{
e_i,e_j\in P_T(x,y)\\
e_i\ne e_j
}}
\zeta(e_i)\zeta(e_j)\langle K_{e_i},K_{e_j}\rangle =d_G(x,y)-d_T(x,y).$$
Hence,
$$\|\phi(y)-\phi(x)\|^2=d_G(x,y).$$
This concludes the proof.
\end{proof}

\section{Quadratic Embedding of $\Theta(1,\beta,\gamma)$} \label{sec.alpha1}

The following observation will be useful to study the quadratic embedding of $\Theta(1,\beta,\gamma)$. Let $G$ be a graph and $G'$ be an arbitrary orientation of $G$. For simplicity, we write $d(x,y)$ instead of $d_G(x,y)$. Consider two edges  $e=(a,b)$ and $e'=(a',b')$  of $G'$, and an inner product
$$\langle e,e'\rangle :=\frac{1}{2}\big(d(a,b')-d(a,a')-d(b,b')+d(b,a')\big). $$
Clearly, $\langle e,e'\rangle$ is symmetric.
The possible values of $\langle e,e'\rangle$ are $0, \pm \frac{1}{2}$ and $\pm1$, see \cite{Winklertree}:
\begin{enumerate}[\rm (i)]
    \item $\langle e,e'\rangle=0$ if
    \begin{align}
        d(a,a')<d(b,a')\quad &\text{and}\quad d(a,b')<d(b,b')\ \text{or} \label{1}\\
         d(b,a')<d(a,a')\quad &\text{and}\quad d(b,b')<d(a,b'),\ \text{or} \label{2}\\
         d(a,a')=d(b,a')\quad &\text{and}\quad d(b,a')=d(b,b'); \label{3}
    \end{align}
    \item  $\langle e,e'\rangle=-\frac{1}{2}$ if
    \begin{align}
         d(a,a')=d(b,a')\quad& \text{and}\quad d(a,b')<d(b,b')\ \text{or} \label{4}\\
         d(a,b')=d(b,b')\quad& \text{and}\quad d(b,a')<d(a,a'); \label{5}
    \end{align}
    \item  $\langle e,e'\rangle=\frac{1}{2}$ if
    \begin{align}
        d(a,b')=d(b,b')\quad& \text{and}\quad d(a,a')<d(b,a')\ \text{or} \label{6}\\
        d(a,a')=d(b,a')\quad& \text{and}\quad d(b,b')<d(a,b'); \label{7}
    \end{align}
    \item  $\langle e,e'\rangle=-1$ if
    \begin{equation}\label{8}
        d(b,a')<d(a,a')\quad \text{and}\quad d(a,b')<d(b,b');
    \end{equation}    
    \item  $\langle e,e'\rangle=1$ if
    \begin{equation} \label{9}
        d(a,a')<d(b,a')\quad \text{and}\quad d(b,b')<d(a,b').
    \end{equation}
\end{enumerate}

\begin{theorem}\label{alpha1}
    Assume that $2\le\beta\le\gamma$.
Then $\Theta(1,\beta,\gamma)$ is of QE class.
\end{theorem}

\begin{proof}
    Consider the graph $\Theta(1,\beta,\gamma)$, with $2\le\beta\le \gamma$. The cases $\beta=2$, $\beta=3$ and $\beta,\ \gamma$ are odd were shown in \cite{MSWtheta}. Hence, we may assume that $\beta\ge 4$.

Let $P_2:=[x_0,x_1]$, $P_{\beta+1}:=[y_0,y_1,\ldots,y_{\beta}]$ and $P_{\gamma+1}:=[z_0,z_1,\ldots,z_{\gamma}]$ be the paths of $\Theta(1,\beta,\gamma)$. 

(i) Let $\beta=2k$ and $\gamma=2l$, where $2\le k \le l $. We endow $\Theta(1,2k,2l)$ with the following orientation (see Figure \ref{orientedtheta1}):
\begin{itemize}
    \item $a_1:=(x_0,x_1)$, where $x_0=y_0=z_0$ and $x_1=y_{2k}=z_{2l}$;
    \item $b_i:=(y_{i-1},y_i)$ for $i=1,\ldots,2k$;
    \item $c_j:=(z_{j-1},z_j)$ for $j=1,\ldots,2l$.
\end{itemize}
To obtain a spanning tree, we remove the edges $a_1$ and $c_{2l}$, see Figure \ref{orientedtree1}.

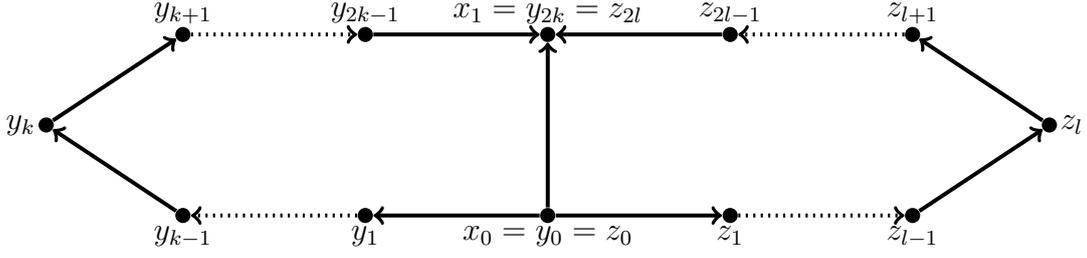
\begin{figure}[h]
\centering
\begin{tikzpicture}[x=12cm,y=12cm]
\tikzset{     
    e4c node/.style={circle,fill=black,minimum size=0.2cm,inner sep=0}, 
    e4c edge/.style={line width=1.5pt}
  }
  \node[e4c node] (a1) at (0.50, 1.00) {};
   \node[anchor= south] at (0.50, 1.00) {$x_1=y_{2k}=z_{2l}$};
  \node[e4c node] (b1) at (0.30, 0.80) {}; 
   \node[anchor= north] at (0.30, 0.80) {$y_1$};
  \node[e4c node] (bk) at (0.10, 1.00) {};
   \node[anchor= south] at (0.10, 1.00) {$y_{k+1}$};
  \node[e4c node] (c1) at (0.70, 0.80) {}; 
   \node[anchor= north] at (0.70, 0.80) {$z_1$};
  \node[e4c node] (ck) at (0.90, 1.00) {};
   \node[anchor= south] at (0.90, 1.00) {$z_{l+1}$};
  \node[e4c node] (c) at (1.05, 0.90) {};
   \node[anchor= west] at (1.05, 0.90) {$z_l$};
  \node[e4c node] (a0) at (0.50, 0.80) {}; 
   \node[anchor= north] at (0.50, 0.80) {$x_0=y_0=z_0$};
  \node[e4c node] (bk-1) at (0.10, 0.80) {};
   \node[anchor= north] at (0.10, 0.80) {$y_{k-1}$};
  \node[e4c node] (b2k-1) at (0.30, 1.00) {}; 
   \node[anchor= south] at (0.30, 1.00) {$y_{2k-1}$};
  \node[e4c node] (ck-1) at (0.90, 0.80) {};
   \node[anchor= north] at (0.90, 0.80) {$z_{l-1}$};
  \node[e4c node] (c2k-1) at (0.70, 1.00) {}; 
   \node[anchor= south] at (0.70, 1.00) {$z_{2l-1}$};
  \node[e4c node] (b) at (-0.05, 0.90) {}; 
   \node[anchor= east] at (-.05, 0.90) {$y_k$};

  \path[->,draw,thick]
  (b) edge[e4c edge]  (bk)
  (bk-1) edge[e4c edge]  (b)
  %(b1) edge[e4c edge]  (bk-1)
  (a0) edge[e4c edge]  (b1)
  (a0) edge[e4c edge]  (a1)
  %(bk) edge[e4c edge]  (b2k-1)
  (b2k-1) edge[e4c edge]  (a1)
  (a0) edge[e4c edge]  (c1)
  %(c1) edge[e4c edge]  (ck-1)
  (ck-1) edge[e4c edge]  (c)
  (c) edge[e4c edge]  (ck)
  %(ck) edge[e4c edge]  (c2k-1)
  (c2k-1) edge[e4c edge]  (a1)
  ;
  
   \path[->,draw,dotted,thick]
(b1) edge[e4c edge] (bk-1)
(bk) edge[e4c edge]  (b2k-1)
(c1) edge[e4c edge]  (ck-1)
(ck) edge[e4c edge]  (c2k-1)
;
\end{tikzpicture}
\caption{Oriented graph $\Theta(1,2k,2l)'$} 
\label{orientedtheta1}
\end{figure}

\begin{figure}[h]
\centering
\begin{tikzpicture}[x=12cm,y=12cm]
\tikzset{     
    e4c node/.style={circle,fill=black,minimum size=0.2cm,inner sep=0}, 
    e4c edge/.style={line width=1.5pt}
  }
  \node[e4c node] (a1) at (0.50, 1.00) {};
   \node[anchor= south] at (0.50, 1.00) {$x_1=y_{2k}=z_{2l}$};
  \node[e4c node] (b1) at (0.30, 0.80) {}; 
   \node[anchor= north] at (0.30, 0.80) {$y_1$};
  \node[e4c node] (bk) at (0.10, 1.00) {};
   \node[anchor= south] at (0.10, 1.00) {$y_{k+1}$};
  \node[e4c node] (c1) at (0.70, 0.80) {}; 
   \node[anchor= north] at (0.70, 0.80) {$z_1$};
  \node[e4c node] (ck) at (0.90, 1.00) {};
   \node[anchor= south] at (0.90, 1.00) {$z_{l+1}$};
  \node[e4c node] (c) at (1.05, 0.90) {};
   \node[anchor= west] at (1.05, 0.90) {$z_l$};
  \node[e4c node] (a0) at (0.50, 0.80) {}; 
   \node[anchor= north] at (0.50, 0.80) {$x_0=y_0=z_0$};
  \node[e4c node] (bk-1) at (0.10, 0.80) {};
   \node[anchor= north] at (0.10, 0.80) {$y_{k-1}$};
  \node[e4c node] (b2k-1) at (0.30, 1.00) {}; 
   \node[anchor= south] at (0.30, 1.00) {$y_{2k-1}$};
  \node[e4c node] (ck-1) at (0.90, 0.80) {};
   \node[anchor= north] at (0.90, 0.80) {$z_{l-1}$};
  \node[e4c node] (c2k-1) at (0.70, 1.00) {}; 
   \node[anchor= south] at (0.70, 1.00) {$z_{2l-1}$};
  \node[e4c node] (b) at (-0.05, 0.90) {}; 
   \node[anchor= east] at (-.05, 0.90) {$y_k$};

  \path[->,draw,thick]
  (b) edge[e4c edge]  (bk)
  (bk-1) edge[e4c edge]  (b)
  %(b1) edge[e4c edge]  (bk-1)
  (a0) edge[e4c edge]  (b1)
  %(a1) edge[e4c edge]  (a0)
  %(bk) edge[e4c edge]  (b2k-1)
  (b2k-1) edge[e4c edge]  (a1)
  (a0) edge[e4c edge]  (c1)
  %(c1) edge[e4c edge]  (ck-1)
  (ck-1) edge[e4c edge]  (c)
  (c) edge[e4c edge]  (ck)
  %(ck) edge[e4c edge]  (c2k-1)
  %(c2k-1) edge[e4c edge]  (a1)
  ;
  
   \path[->,draw,dotted,thick]
(b1) edge[e4c edge] (bk-1)
(bk) edge[e4c edge]  (b2k-1)
(c1) edge[e4c edge]  (ck-1)
(ck) edge[e4c edge]  (c2k-1)
;
\end{tikzpicture}
\caption{Spanning tree of $\Theta(1,2k,2l)'$} 
\label{orientedtree1}
\end{figure}
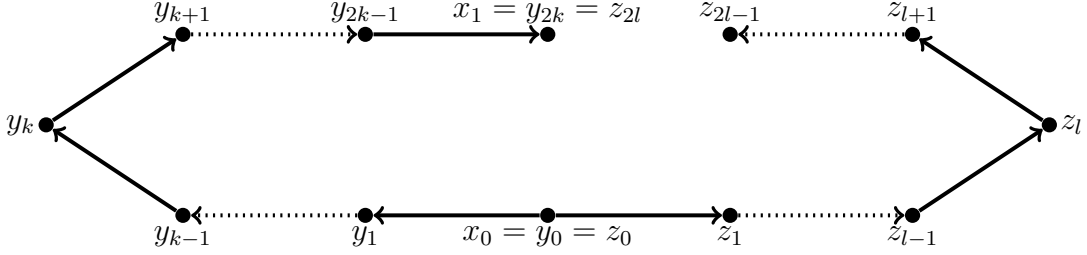

First, we calculate the value of $\langle b_p,b_r\rangle$ for $p,r\in \{1,\ldots,2k\}$. Clearly,
$$\langle b_p,b_p\rangle=1,\quad \quad p=1,\ldots,2k. $$
Note that the distance between a pair of vertices from the $P_{\beta+1}$ is at most $k$, because the path $P_{\beta+1}$ together with the path $P_2$ creates the cycle with $2k+1$ vertices. For $i=1,\ldots,k$, consider the edges $b_i=(y_{i-1},y_i)$, $b_{i+k}=(y_{i+k-1},y_{i+k})$ and, if $i<k$, $b_{i+k+1}=(y_{i+k},y_{i+k+1})$. Then $d(y_{i-1},y_{i+k-1})=k>d(y_i,y_{i+k-1})=k-1$ and $d(y_{i-1},y_{i+k})=i-1+1+2k-(i-k)=k=d(y_i,y_{i+k})$, hence by (\ref{5}) we obtain $$\langle b_i,b_{i+k}\rangle=-\frac{1}{2},\quad \quad i=1,\ldots,k.$$ 

Similarly we calculate $\langle b_i,b_{i+k+1}\rangle$. We have $d(y_{i-1},y_{i+k+1})=i-1+1+2k-(i-k+1)=k-1$ and $d(y_i,y_{i+k+1})=i+1+2k-(i-k+1)=k$, hence by (\ref{4}) we obtain
$$\langle b_i,b_{i+k+1}\rangle=-\frac{1}{2},\quad\quad i=1,\ldots,k-1.$$

Let $0<j<k$ and consider the edge $b_{i+j}=(y_{i+j-1},y_{i+j})$. Then $d(y_{i-1},y_{i+j-1})=j> d(y_{i},y_{i+j-1})=j-1$ and 
$d(y_{i-1},y_{i+j})=j+1> d(y_{i},y_{i+j})=j$, hence by (\ref{2}) we obtain 
$$\langle b_i, b_{i+j}\rangle=0,\quad \quad i=1,\ldots,k,\ j=1,\ldots,k-1. $$
In similar manner can be proven that $\langle b_i, b_{i+j}\rangle=0 $ for $j>k$ and $i+j\le 2k$. 

Since the inner product $\langle e,e'\rangle$ is symmetric, for $p,r\in \{1,\ldots,2k\}$ we obtain
$$\langle b_p, b_{r}\rangle=\begin{cases}
1 & \text{if}\ p=r\\
-\frac{1}{2} & \text{if}\ |p-r|=k\ \text{or}\ |p-r|=k+1\\
0 & \text{otherwise}
\end{cases}. $$

Inasmuch as the path $P_{\gamma+1}$ has the same parity as $P_{\beta+1}$, for $\langle c_s, c_t\rangle$, where $s,t\in \{1,\ldots,2l\}$, we obtain
$$\langle c_s, c_t\rangle=\begin{cases}
1 & \text{if}\ s=t\\
-\frac{1}{2} & \text{if}\ |c-t|=l\ \text{or}\ |c-t|=l+1\\
0 & \text{otherwise}
\end{cases}. $$

Now, consider the inner product $\langle b_p,c_s\rangle$ for $p=1,\ldots,2k$ and $s=1,\ldots,2l-1$, because the edge $c_{2l}$ is removed. Let $b_i=(y_{i-1},y_i)$ for $i\in\{1,\ldots,k-1\}\cup\{k+2,\ldots,2k\}$ and $z_j$ for $j=0,\ldots,2l-1$. Then $d(y_{i-1},z_j)<d(y_i,z_j)$ for $i=1,\ldots,k-1$ and $d(y_{i-1},z_j)>d(y_i,z_j)$ for $i=k+2,\ldots,2k$. Therefore by (\ref{1}) and (\ref{2}), respectively, we obtain
$$\langle b_i,c_j\rangle=0,\quad \quad i\in\{1,\ldots,k-1\}\cup\{k+2,\ldots,2k\},\ j=0,\ldots,2l.$$

The remaining edges to consider are $b_k$ and $b_{k+1}$. We have $d(y_{k-1},z_j)=k-1+j<d(y_k,z_j)=k+j$ for $j=0,\ldots,l$. By (\ref{1}) we obtain
$$\langle b_k,c_j \rangle=0,\quad \quad j=1,\ldots,l.$$
Whereas for $j\in \{l+1,\ldots,2l\}$ we have $d(y_{k-1},z_j)=d(y_k,z_j)=k+2l-j$. By (\ref{3}) we obtain 
$$\langle b_k,c_j \rangle=0, \quad\quad j=l+2,\ldots,2l-1.$$
Note that $d(y_{k-1},z_l)=k-1+l<d(y_k,z_l)=k+l$ and $d(y_{k-1},z_{l+1})=d(y_k,z_{l+1})=k+l-1$, hence (\ref{6}) implies that
$$\langle b_k,c_{l+1}\rangle=\frac{1}{2}.$$
Symmetric argument leads to $\langle b_{k+1},c_l\rangle=\frac{1}{2}.$ Hence the inner product $\langle b_p,c_s\rangle$ for $p=1,\ldots,2k$ and $s=1,\ldots,2l-1$ has the form
$$\langle b_p, c_s\rangle=\begin{cases}
\frac{1}{2} & \text{if}\ i=k,j=l+1\ \text{or}\ i=k+1,j=l\\
0 & \text{otherwise}
\end{cases}. $$

Now, we summarize all non-zero values of forms. Recall that, $$\langle b_p,b_p\rangle=\langle c_s,c_s\rangle=1, \quad\quad p=1,\ldots,2k,\ s=1,\ldots,2l-1.$$
In other cases, we have 
\begin{align*}
  &\langle b_p,b_{p+k}\rangle=\langle b_p,b_{p+k+1}\rangle=-\frac{1}{2}, \quad\quad p=1,\ldots,k-1;\\
  &\langle b_k,b_{2k}\rangle=-\frac{1}{2},\ \langle b_k, c_{l+1}\rangle=\frac{1}{2};\\
  &\langle b_{k+1},b_1\rangle=-\frac{1}{2},\ \langle b_{k+1}, c_l\rangle=\frac{1}{2};\\
  &\langle b_p,b_{p-k}\rangle=\langle b_p,b_{p-k-1}\rangle=-\frac{1}{2},\quad\quad p=k+2,\ldots,2k;\\
  & \langle c_s,c_{s+l}\rangle=\langle c_s,c_{s+l+1}\rangle=-\frac{1}{2},\quad\quad \ s=1,\ldots,l-2;\\
  & \langle c_{l-1}, c_{2l-1}\rangle=-\frac{1}{2};\\
  & \langle c_l,b_{k+1}\rangle=\frac{1}{2};\\
  & \langle c_{l+1},b_k\rangle=\frac{1}{2},\ \langle c_{l+1}, c_1\rangle=-\frac{1}{2};
\end{align*}
and finally
$$\langle c_s,c_{s-l}\rangle=\langle c_s,c_{s-l-1}\rangle=-\frac{1}{2},\quad\quad s=l+2,\ldots,2l-1.$$
Recall that the RKHS is on the set of edges $E(\Theta(1,2k,2l)')$, while $c_{2l}\notin E(\Theta(1,2k,2l)')$ by construction.

Now, we define:
    \begin{itemize}
        \item $e_i:=b_i$ for $i=1,\ldots,2k$;
        \item $e_{2k+j}:=c_j$ for $j=1,\ldots,2l-1$ 
    \end{itemize}
    and define a matrix $M:=[\langle e_r,e_t\rangle]_{r,t=1,\ldots,2k+2l-1}$. Then

\[
M=\begin{bmatrix}
    I_k & A & 0 & C\\
    A^\top & I_k & D & 0\\
    0 & D^\top & I_l & B\\
    C^\top & 0 & B^\top & I_{l-1} 
\end{bmatrix},
\]  
where
\[
A=
\begin{bmatrix}
-\frac12 & -\frac12 &  &  & \\
 & -\frac12 & \ddots &  & \\
 &  & \ddots & \ddots & \\
 &  &  & -\frac12 & -\frac12\\
 & &  &  & -\frac12
\end{bmatrix}_{k\times k},\quad
B=\begin{bmatrix}
  -\frac12 & -\frac12 &  &  \\
 & -\frac12 & \ddots &  \\
 &  & \ddots & \ddots & \\
 &  &  & -\frac12 & -\frac12\\
 & &  &  & -\frac12\\
 & & & &
\end{bmatrix}_{l\times l-1},
\]
\[
C=\begin{bmatrix}
    & & & &\\
    & & & &\\
    & & & &\\
    \frac12 & & & &
\end{bmatrix}_{k\times l-1},\quad \text{and}\quad
D=\begin{bmatrix}
    & & & & &\frac12 \\
    & & & & &\\
    & & & & &\\
    & & & & &
\end{bmatrix}_{k\times l}.
\]
Blank entries represent zeros. Alternatively, matrices $A,B,C$ and $D$ can be defined by the following entry-wise formulas:
\begin{align*}
A=[a_{ij}]_{1\le i,j\le k},\quad&
a_{ij}=
\begin{cases}
-\tfrac12 & \text{if } j=i \text{ or } j=i+1,\\
0 & \text{otherwise}
\end{cases}\\
B=[b_{ij}]_{1\le i \le l,\ 1\le j\le l-1},\quad&
b_{ij}=
\begin{cases}
-\tfrac12 & \text{if } j=i \text{ or } j=i+1,\\
0 & \text{otherwise}
\end{cases}\\
C=[c_{ij}]_{1\le i \le k,\ 1\le j\le l-1},\quad&
c_{ij}=
\begin{cases}
\tfrac12 &\ \ \text{if }\ i=k\ \text{and}\ j=1,\\
0 & \ \ \text{otherwise}.
\end{cases}\\
D=[d_{ij}]_{1\le i \le k,\ 1\le j\le l},\quad&
d_{ij}=
\begin{cases}
\tfrac12 &\ \ \text{if }\ i=1\ \text{and}\ j=l,\\
0 & \ \ \text{otherwise}.
\end{cases}
\end{align*}

    Note that each diagonal element of $M$ equals 1 and the sum of the absolute values of the non-diagonal entries in each row is at most 1. Hence, all eigenvalues of $M$ belong to the interval $[0,2]$, see \cite{HorJ85}. Thus $M$ is positive semidefinite, and by Theorem \ref{tree} the theta graph $\Theta(1,\beta,\gamma)$ is of QE class for even $\beta$ and $\gamma$.

    (ii) Let $\beta=2k$ and $\gamma=2l+1$, where $2\le k,l $. We endow $\Theta(1,2k,2l+1)$ with the following orientation (see Figure \ref{orientedtheta2}):
\begin{itemize}
    \item $a_1:=(x_0,x_1)$, where $x_0=y_0=z_0$ and $x_1=y_{2p}=z_{2r+1}$;
    \item $b_i:=(y_{i-1},y_i)$ for $i=1,\ldots,2k$;
    \item $c_j:=(z_{j-1},z_j)$ for $j=1,\ldots,2l+1$.
\end{itemize}
To obtain a spanning tree, we remove the edges $a_1$ and $c_{l+1}$, see Figure \ref{orientedtree2}.

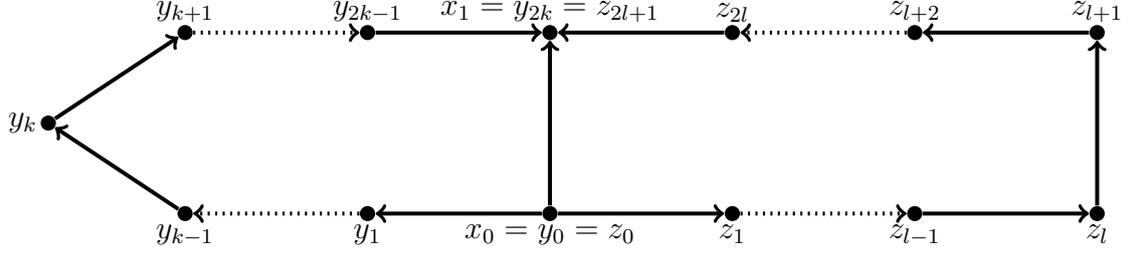
\begin{figure}[h]
\centering
\begin{tikzpicture}[x=12cm,y=12cm]
\tikzset{     
    e4c node/.style={circle,fill=black,minimum size=0.2cm,inner sep=0}, 
    e4c edge/.style={line width=1.5pt}
  }
  \node[e4c node] (a1) at (0.50, 1.00) {};
   \node[anchor= south] at (0.50, 1.00) {$x_1=y_{2k}=z_{2l+1}$};
  \node[e4c node] (b1) at (0.30, 0.80) {}; 
   \node[anchor= north] at (0.30, 0.80) {$y_1$};
  \node[e4c node] (bk) at (0.10, 1.00) {};
   \node[anchor= south] at (0.10, 1.00) {$y_{k+1}$};
  \node[e4c node] (c1) at (0.70, 0.80) {}; 
   \node[anchor= north] at (0.70, 0.80) {$z_1$};
  \node[e4c node] (ck) at (0.90, 1.00) {};
   \node[anchor= south] at (0.90, 1.00) {$z_{l+2}$};
  \node[e4c node] (c) at (1.10, 0.80) {};
   \node[anchor= north] at (1.10, 0.80) {$z_l$};
  \node[e4c node] (a0) at (0.50, 0.80) {}; 
   \node[anchor= north] at (0.50, 0.80) {$x_0=y_0=z_0$};
  \node[e4c node] (bk-1) at (0.10, 0.80) {};
   \node[anchor= north] at (0.10, 0.80) {$y_{k-1}$};
  \node[e4c node] (b2k-1) at (0.30, 1.00) {}; 
   \node[anchor= south] at (0.30, 1.00) {$y_{2k-1}$};
  \node[e4c node] (ck-1) at (0.90, 0.80) {};
   \node[anchor= north] at (0.90, 0.80) {$z_{l-1}$};
  \node[e4c node] (c2k-1) at (0.70, 1.00) {}; 
   \node[anchor= south] at (0.70, 1.00) {$z_{2l}$};
  \node[e4c node] (b) at (-0.05, 0.90) {}; 
   \node[anchor= east] at (-.05, 0.90) {$y_k$};
    \node[e4c node] (z) at (1.10, 1.00) {};
   \node[anchor= south] at (1.10, 1.00) {$z_{l+1}$};

  \path[->,draw,thick]
  (b) edge[e4c edge]  (bk)
  (bk-1) edge[e4c edge]  (b)
  %(b1) edge[e4c edge]  (bk-1)
  (a0) edge[e4c edge]  (b1)
  (a0) edge[e4c edge]  (a1)
 % (bk) edge[e4c edge]  (b2k-1)
  (b2k-1) edge[e4c edge]  (a1)
  (a0) edge[e4c edge]  (c1)
 % (c1) edge[e4c edge]  (ck-1)
  (ck-1) edge[e4c edge]  (c)
  %(c) edge[e4c edge]  (ck)
  %(ck) edge[e4c edge]  (c2k-1)
  (c2k-1) edge[e4c edge]  (a1)
  (c) edge[e4c edge]  (z)
  (z) edge[e4c edge]  (ck)
  ;

  \path[->,draw,dotted,thick]
(b1) edge[e4c edge] (bk-1)
(bk) edge[e4c edge]  (b2k-1)
(ck) edge[e4c edge]  (c2k-1)
(c1) edge[e4c edge]  (ck-1)
;
\end{tikzpicture}
\caption{Oriented graph $\Theta(1,2k,2l+1)'$} 
\label{orientedtheta2}
\end{figure}

\begin{figure}[h]
\centering
\begin{tikzpicture}[x=12cm,y=12cm]
\tikzset{     
    e4c node/.style={circle,fill=black,minimum size=0.2cm,inner sep=0}, 
    e4c edge/.style={line width=1.5pt}
  }
  \node[e4c node] (a1) at (0.50, 1.00) {};
   \node[anchor= south] at (0.50, 1.00) {$x_1=y_{2k}=z_{2l+1}$};
  \node[e4c node] (b1) at (0.30, 0.80) {}; 
   \node[anchor= north] at (0.30, 0.80) {$y_1$};
  \node[e4c node] (bk) at (0.10, 1.00) {};
   \node[anchor= south] at (0.10, 1.00) {$y_{k+1}$};
  \node[e4c node] (c1) at (0.70, 0.80) {}; 
   \node[anchor= north] at (0.70, 0.80) {$z_1$};
  \node[e4c node] (ck) at (0.90, 1.00) {};
   \node[anchor= south] at (0.90, 1.00) {$z_{l+2}$};
  \node[e4c node] (c) at (1.10, 0.80) {};
   \node[anchor= north] at (1.10, 0.80) {$z_l$};
  \node[e4c node] (a0) at (0.50, 0.80) {}; 
   \node[anchor= north] at (0.50, 0.80) {$x_0=y_0=z_0$};
  \node[e4c node] (bk-1) at (0.10, 0.80) {};
   \node[anchor= north] at (0.10, 0.80) {$y_{k-1}$};
  \node[e4c node] (b2k-1) at (0.30, 1.00) {}; 
   \node[anchor= south] at (0.30, 1.00) {$y_{2k-1}$};
  \node[e4c node] (ck-1) at (0.90, 0.80) {};
   \node[anchor= north] at (0.90, 0.80) {$z_{l-1}$};
  \node[e4c node] (c2k-1) at (0.70, 1.00) {}; 
   \node[anchor= south] at (0.70, 1.00) {$z_{2l}$};
  \node[e4c node] (b) at (-0.05, 0.90) {}; 
   \node[anchor= east] at (-.05, 0.90) {$y_k$};
    \node[e4c node] (z) at (1.10, 1.00) {};
   \node[anchor= south] at (1.10, 1.00) {$z_{l+1}$};

  \path[->,draw,thick]
  (b) edge[e4c edge]  (bk)
  (bk-1) edge[e4c edge]  (b)
  %(b1) edge[e4c edge]  (bk-1)
  (a0) edge[e4c edge]  (b1)
  %(a0) edge[e4c edge]  (a1)
 % (bk) edge[e4c edge]  (b2k-1)
  (b2k-1) edge[e4c edge]  (a1)
  (a0) edge[e4c edge]  (c1)
 % (c1) edge[e4c edge]  (ck-1)
  (ck-1) edge[e4c edge]  (c)
  %(c) edge[e4c edge]  (ck)
  %(ck) edge[e4c edge]  (c2k-1)
  (c2k-1) edge[e4c edge]  (a1)
  %(c) edge[e4c edge]  (z)
  (z) edge[e4c edge]  (ck)
  ;

  \path[->,draw,dotted,thick]
(b1) edge[e4c edge] (bk-1)
(bk) edge[e4c edge]  (b2k-1)
(ck) edge[e4c edge]  (c2k-1)
(c1) edge[e4c edge]  (ck-1)
;
\end{tikzpicture}
\caption{Spanning tree of $\Theta(1,2k,2l+1)'$} 
\label{orientedtree2}
\end{figure}
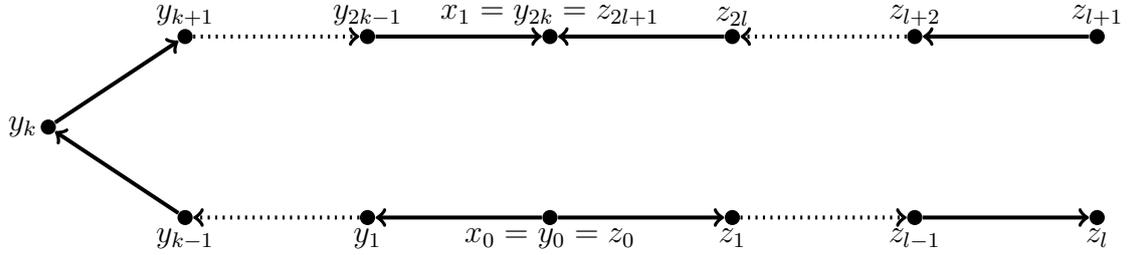

In the same way as in case (i), for $p,r\in \{1,\ldots,2k\}$ we obtain
$$\langle b_p, b_{r}\rangle=\begin{cases}
1 & \text{if}\ p=r\\
-\frac{1}{2} & \text{if}\ |p-r|=k\ \text{or}\ |p-r|=k+1\\
0 & \text{otherwise}
\end{cases}. $$

Now, consider the form between edges $b_p$ and $c_s$ for $p=1,\ldots,2k$ and $s\in \{1,\ldots,l\}\cup \{l+2,\ldots,2l+1\}$, because the edge $c_{l+1}$ is removed. Let $b_i=(y_{i-1},y_i)$ for $i\in \{1,\ldots,k-1\} \cup \{k+2,\ldots,2k\}$ and $z_j$ for $j=1,\ldots,2l+1$. Then we have $d(y_{i-1},z_j)<d(y_i,z_j)$ for $i=1,\ldots,k-1$ and $d(y_{i-1},z_j)>d(y_i,z_j)$ for $i=k+2,\ldots,2k$. Therefore, (\ref{1}) and (\ref{2}), respectively, imply that
$$\langle b_i,c_j\rangle=0,\quad\quad i\in \{1,\ldots,k-1\} \cup \{k+2,\ldots,2k\}, \ j\in \{1,\ldots,l\}\cup \{l+2,\ldots,2l+1\}. $$

We need to verify the relation between $b_k$ and $c_j$, as well as $b_{k+1}$ and $c_j$ for $j\in \{1,\ldots,l\}\cup \{l+2,\ldots,2l+1\}$. Clearly, $d(y_{k-1},z_j)=k-1+j<d(y_k,z_j)=k+j$ for $j\in \{0,\ldots,l\}$, thus (\ref{1}) implies that $\langle b_k,c_j\rangle=0$ for $j\in \{1,\ldots,l\}$. Whereas for $j\in \{l+1,\ldots,2l+1\}$ we have $d(y_{k-1},z_j)=d(y_k,z_j)=k+2l-j$, thus (\ref{3}) implies that $\langle b_k,c_j\rangle=0$ for $j\in \{l+2,\ldots,2l+1\}$. Since the edge $c_{l+1}$ does not belong to the tree, we obtain
$$\langle b_k,c_j\rangle=0, \quad\quad  j\in \{1,\ldots,l\}\cup \{l+2,\ldots,2l+1\}.$$
In a similar manner, we can show that 
$$\langle b_{k+1},c_j\rangle=0, \quad\quad j\in \{1,\ldots,l\}\cup \{l+2,\ldots,2l+1\}. $$
Hence, we obtain
$$\langle b_p,c_s\rangle=0,\quad \quad p=1,\ldots,2k, \  s\in \{1,\ldots,l\}\cup \{l+2,\ldots,2l+1\}.$$

Now, consider the relation between edges $c_s$ and $c_t$ for $s,t\in \{1,\ldots,l\}\cup \{l+2,\ldots,2l+1\}$, because the edge $c_{l+1}$ is removed. Clearly,
$$\langle c_s,c_s\rangle=1,\quad\quad s\in \{1,\ldots,l\}\cup \{l+2,\ldots,2l+1\}. $$
Note that the distance between each pair of two vertices from $P_{\gamma+1}$ is at most $l+1$, because the paths $P_{\gamma+1}$ and $P_{2}$ create the cycle on $2l+2$ vertices. Consider the edges $c_i=(z_{i-1},z_i)$ and $c_{i+l+1}=(z_{i+l},z_{i+l+1})$ for $i=1,\ldots,l$. We have $d(z_{i-1},z_{i+l})=l+1>d(z_i,z_{i+l})=l$ and $d(z_{i-1},z_{i+l+1})=i+2l-(i+l+1)=l<d(z_i,z_{i+l+1})=l+1$. Hence (\ref{8}) implies that 
$$\langle c_i,c_{i+l+1}\rangle=-1, \quad\quad i=1,\ldots,l. $$

Let $0<j<l$ and consider the edge $b_{i+j+1}=(y_{i+j},y_{i+j+1})$. We have $d(z_{i-1},z_{i+j})=j+1>d(z_i,z_{i+j})=j$ and $d(z_{i-1},z_{i+j+1})=j+2>d(z_i,z_{i+j+1})=j+1$. Hence (\ref{2}) implies that 
$$\langle c_i, c_{i+j+1}\rangle=0, \quad\quad i=1,\ldots,l,\ j=1,\ldots,l-1.  $$
In similar manner we can show that $\langle c_i, c_{i+j+1}\rangle=0 $ for $j>k$ and $i+j\le 2l$.

Since the relation $\langle e,e'\rangle$ is symmetric, for $s,t\in \{1,\ldots,l\}\cup \{l+2,\ldots,2l+1\}$ we obtain
$$\langle c_s, c_t\rangle=\begin{cases}
1 & \text{if}\ s=t\\
-1 & \text{if}\ |s-t|=l+1\\
0 & \text{otherwise}
\end{cases}. $$

Now, we summarize all non-zero values of relations. Clearly, $$\langle b_p,b_p\rangle=\langle c_s,c_s\rangle=1, \quad\quad p=1,\ldots,2k,\ s\in \{1,\ldots,l\}\cup \{l+2,\ldots,2l+1\}.$$
In other cases, we have 
\begin{align*}
   &\langle b_p,b_{p+k}\rangle=\langle b_p,b_{p+k+1}\rangle=-\frac{1}{2},\quad p=1,\ldots,k-1;\\
   &\langle b_k,b_{2k}\rangle=-\frac{1}{2};\\
   &\langle b_{k+1},b_1\rangle=-\frac{1}{2};\\
   &\langle b_p,b_{p-k}\rangle=\langle b_p,b_{p-k-1}\rangle=-\frac{1}{2},\quad p=k+2,\ldots,2k;\\
   &\langle c_s,c_{s+l+1}\rangle=-1,\quad\quad\quad\quad\quad\quad\ \ s=1,\ldots,l;\\
\end{align*}
and finally
$$\langle c_s,c_{s-l-1}\rangle=-1, \quad\quad s=l+2,\ldots,2l+1.$$
Recall that the RKHS is on the set of edges $E(\Theta(1,2k,2l+1)')$, while $c_{l+1}\notin E(\Theta(1,2k,2l+1)')$ by construction.

Now, we define:
    \begin{itemize}
        \item $e_h:=b_h$ for $h=1,\ldots,2k$;
        \item $e_{2k+i}:=c_i$ for $i=1,\ldots,l$ 
        \item $e_{2k+j-1}:=c_j$ for $j=l+2,\ldots,2l+1$
    \end{itemize}
    and define a matrix $M:=[\langle e_r,e_t\rangle]_{r,t=1,\ldots,2k+2l}$. Then

\[
M=\begin{bmatrix}
  B & 0 \\
  0 & C
\end{bmatrix},
\]
where
\[
B=\begin{bmatrix}
    I_k & A\\
    A^\top & I_k
\end{bmatrix},\quad
A=
\begin{bmatrix}
-\frac12 & -\frac12 &  &  & \\
 & -\frac12 & \ddots &  & \\
 &  & \ddots & \ddots & \\
 &  &  & -\frac12 & -\frac12\\
 & &  &  & -\frac12
\end{bmatrix}_{k\times k},\quad \text{and}\quad
C=\begin{bmatrix}
    I_l & -I_l\\
    -I_l & I_l
\end{bmatrix}.
\]
Clearly, the matrix $A$ is the same as in the case (i).

    Note that each diagonal element of $M$ equals 1 and the sum of the absolute values of the non-diagonal entries in each row is at most 1. Hence, all eigenvalues of $M$ belong to the interval $[0,2]$, see \cite{HorJ85}. Thus $M$ is positive semidefinite, and by Theorem \ref{tree} the theta graph $\Theta(1,\beta,\gamma)$ is of QE class for even $\beta$ and odd $\gamma$ or odd $\beta$ and even $\gamma$. This concludes the proof of Theorem~\ref{alpha1}.
\end{proof}

We conclude this Section with the following corollary.

\begin{corollary}
Let $2\le \beta \le \gamma$. If at least one of $\beta$ and $\gamma$ is odd, then $\QEC(\Theta(1,\beta,\gamma))=~0$. If $\beta$ and $\gamma$ are even, then $\QEC(1,\beta,\gamma)\in [-\frac{1}{4\cos^2\frac{\pi}{\gamma+1}}, 0]$.   
\end{corollary}
\begin{proof}
The graph $\Theta(1,\beta,\gamma)$ is of QE class, hence $\QEC(\Theta(1,\beta,\gamma))\le 0$. 

If at least one of $\beta$ and $\gamma$ is odd, then $\Theta(1,\beta,\gamma)$ contains a cycle on an even number of vertices as an isometrically embedded subgraph, such a graph has a QE constant equal to 0, see \cite{Obata3}. Hence $\QEC(\Theta(1,\beta,\gamma))=0$.

If $\beta$ and $\gamma$ are even, then $\Theta(1,\beta,\gamma)$ contains two cycles, both of an odd number of vertices $\beta+1$ and $\gamma+1$ respectively, as isometrically embedded subgraphs. For these cycles, we have $\QEC(C_{\beta+1})=-\frac{1}{4\cos^2\frac{\pi}{\beta+1}}$ and $\QEC(C_{\gamma+1})=-\frac{1}{4\cos^2\frac{\pi}{\gamma+1}}$, see \cite{Obata3}. Clearly, $-\frac{1}{4\cos^2\frac{\pi}{\beta+1}}\le -\frac{1}{4\cos^2\frac{\pi}{\gamma+1}}$. Hence, for even $\beta$ and $\gamma$ we obtain $\QEC(1,\beta,\gamma)\ge-\frac{1}{4\cos^2\frac{\pi}{\gamma+1}}$, which concludes the proof.
\end{proof}

\section{Quadratic Embedding of $\Theta(2,3,\gamma)$} \label{sec.alpha2}

In this section, we prove the following theorem to achieve the characterization of theta graphs according to the QE property.

\begin{theorem}\label{alpha2}
    Assume that $\gamma\ge 3$.
Then $\Theta(2,3,\gamma)$ is of QE class if and only if either $\gamma\in \{3,5,7\}$.
\end{theorem}
\begin{proof}
The non-embeddability of graphs $\Theta(2,3,\gamma)$ for even $\gamma$ was shown in \cite{MSWtheta}, hence let $\gamma\ge3$ be odd. 

The proof is divided into two parts. First, we apply Theorem \ref{tree} to show the embeddability of graphs $\Theta(2,3,3)$, $\Theta(2,3,5)$, and $\Theta(2,3,7)$. In the second part, we present the vectors which show that the distance matrices of graphs $\Theta(2,3,2k+7)$ for $k\ge1$ are not conditionally negative definite.

(i) Consider the graph $\Theta(2,3,3)$ and its oriented spanning tree, see Figure \ref{233}.

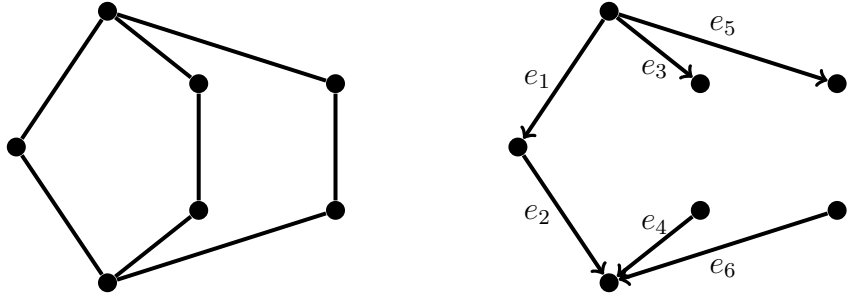
\begin{figure}[h]
\centering
\begin{tikzpicture}[x=12cm,y=12cm]
\tikzset{     
    e4c node/.style={circle,fill=black,minimum size=0.25cm,inner sep=0}, 
    e4c edge/.style={line width=1.5pt}
  }
  \node[e4c node] (1) at (0.15, 0.85) {}; 
  \node[e4c node] (2) at (0.05, 0.70) {}; 
  \node[e4c node] (3) at (0.15, 0.55) {}; 
  \node[e4c node] (4) at (0.25, 0.77) {}; 
  \node[e4c node] (5) at (0.25, 0.63) {}; 
  \node[e4c node] (6) at (0.40, 0.77) {}; 
  \node[e4c node] (7) at (0.40, 0.63) {}; 
  \node[e4c node] (8) at (0.70, 0.85) {}; 
  \node[e4c node] (9) at (0.60, 0.70) {}; 
  \node[e4c node] (10) at (0.70, 0.55) {}; 
  \node[e4c node] (11) at (0.80, 0.77) {}; 
  \node[e4c node] (12) at (0.80, 0.63) {}; 
  \node[e4c node] (13) at (0.95, 0.77) {}; 
  \node[e4c node] (14) at (0.95, 0.63) {};

  \path[draw,thick]
  (1) edge[e4c edge]  (2)
  (7) edge[e4c edge]  (3)
  (6) edge[e4c edge]  (7)
  (1) edge[e4c edge]  (6)
  (5) edge[e4c edge]  (3)
  (4) edge[e4c edge]  (5)
  (1) edge[e4c edge]  (4)
  (2) edge[e4c edge]  (3)
  ;

   \path[->,draw,thick]
  (8) edge[e4c edge] node [anchor= east] {$e_1$} (9)
  (9) edge[e4c edge] node [anchor= east] {$e_2$} (10)
  (8) edge[e4c edge] node [anchor= north] {$e_3$} (11)
  (12) edge[e4c edge] node [anchor= south] {$e_4$} (10)
  (8) edge[e4c edge] node [anchor= south] {$e_5$} (13)
  (14) edge[e4c edge] node [anchor= north] {$e_6$} (10)
  ;
\end{tikzpicture}

\caption{The graph $\Theta(2,3,3)$ and its oriented spanning tree} \label{233}
\end{figure}

According to Winkler's theorem, the matrix $2M=[2\langle e_i,e_j\rangle]_{i,j=1,\ldots,6}$ of the graph $\Theta(2,3,3)$ has the form
\[2M=\begin{bmatrix}
2 & 0 & 0 & 1 & 0 & 1 \\
0 & 2 & 1 & 0 & 1 & 0 \\
0 & 1 & 2 & -1 & 0 & 1 \\
1 & 0 & -1 & 2 & 1 & 0 \\
0 & 1 & 0 & 1 & 2 & -1\\
1 & 0 & 1 & 0 & -1 & 2
\end{bmatrix}.
\]
The characteristic polynomial of $2M$ has the form
$$\det(2M-\lambda I)=\lambda^6-12\lambda^5+52\lambda^4-96\lambda^3+68\lambda^2-16\lambda=\lambda(\lambda-4)(\lambda-2-\sqrt{2})^2(\lambda-2+\sqrt{2})^2$$
Therefore, the spectrum of $2M$ is given by
$$\sigma(2M)=\bigg\{4,2+\sqrt{2},2+\sqrt{2},2-\sqrt{2},2-\sqrt{2},0\bigg\}.$$
All eigenvalues of $2M$ are non-negative, hence $M$ is positive semidefinite, and by Theorem \ref{tree}, the graph $\Theta(2,3,3)$ is of QE class.

Similarly, consider the graph $\Theta(2,3,5)$ and its oriented spanning tree, see Figure \ref{235}.

\begin{figure}[H]
\centering
\begin{tikzpicture}[x=12cm,y=12cm]
\tikzset{     
    e4c node/.style={circle,fill=black,minimum size=0.25cm,inner sep=0}, 
    e4c edge/.style={line width=1.5pt}
  }
  \node[e4c node] (1) at (0.15, 0.85) {}; 
  \node[e4c node] (2) at (0.05, 0.70) {}; 
  \node[e4c node] (3) at (0.15, 0.55) {}; 
  \node[e4c node] (4) at (0.25, 0.77) {}; 
  \node[e4c node] (5) at (0.25, 0.63) {}; 
  \node[e4c node] (6) at (0.30, 0.80) {};
  \node[e4c node] (7) at (0.45, 0.75) {};
  \node[e4c node] (8) at (0.45, 0.65) {};
 \node[e4c node] (9) at (0.30, 0.60) {};

   \node[e4c node] (11) at (0.75, 0.85) {}; 
  \node[e4c node] (12) at (0.65, 0.70) {}; 
  \node[e4c node] (13) at (0.75, 0.55) {}; 
  \node[e4c node] (14) at (0.85, 0.77) {}; 
  \node[e4c node] (15) at (0.85, 0.63) {}; 
  \node[e4c node] (16) at (0.90, 0.80) {};
  \node[e4c node] (17) at (1.05, 0.75) {};
  \node[e4c node] (18) at (1.05, 0.65) {};
 \node[e4c node] (19) at (0.90, 0.60) {};

  \path[draw,thick]
  (1) edge[e4c edge]  (2)
  (1) edge[e4c edge]  (4)
  (2) edge[e4c edge]  (3)
  (4) edge[e4c edge]  (5)
  (5) edge[e4c edge]  (3)
  (1) edge[e4c edge]  (6)
  (6) edge[e4c edge]  (7)
  (7) edge[e4c edge]  (8)
  (8) edge[e4c edge]  (9)
  (9) edge[e4c edge]  (3)
  ;

   \path[->,draw,thick]
  (11) edge[e4c edge] node [anchor= east] {$e_1$} (12)
  (11) edge[e4c edge] node [anchor= north] {$e_3$} (14)
  (12) edge[e4c edge] node [anchor= east] {$e_2$} (13)
  %(14) edge[e4c edge]  (15)
  (15) edge[e4c edge] node [anchor= south] {$e_4$} (13)
  (11) edge[e4c edge] node [anchor= south] {$e_5$} (16)
  (16) edge[e4c edge] node [anchor= south] {$e_6$} (17)
  %(17) edge[e4c edge]  (18)
  (18) edge[e4c edge] node [anchor= north] {$e_7$} (19)
  (19) edge[e4c edge] node [anchor= north] {$e_8$} (13)
  ;
  ;
\end{tikzpicture}

\caption{The graph $\Theta(2,3,5)$ and its oriented spanning tree} \label{235}
\end{figure}
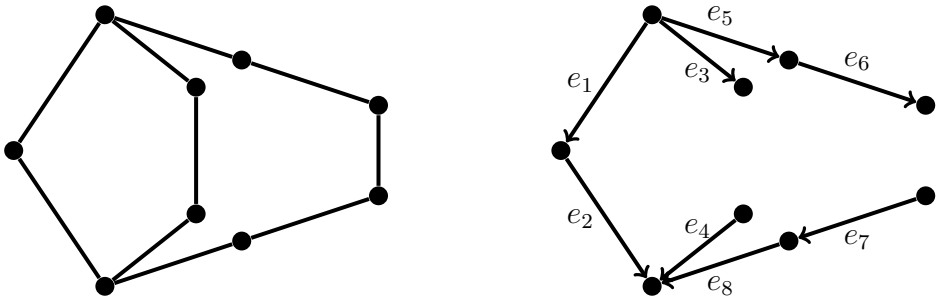

Now, the matrix $2M=[2\langle e_i,e_j\rangle]_{i,j=1,\ldots,8}$ of the graph $\Theta(2,3,5)$ has the form
\[2M=\begin{bmatrix}
2 & 0 & 0 & 1 & 0 & 0 & 1 & 0 \\
0 & 2 & 1 & 0 & 0 & 1 & 0 & 0 \\
0 & 1 & 2 & -1 & 0 & 0 & 1 & 0 \\
1 & 0 & -1 & 2 & 0 & 1 & 0 & 0 \\
0 & 0 & 0 & 0 & 2 & 0 & -1 & -1\\
0 & 1 & 0 & 1 & 0 & 2 & 0 & -1\\
1 & 0 & 1 & 0 & -1 & 0 & 2 & 0\\
0 & 0 & 0 & 0 & -1 & -1 & 0 & 2
\end{bmatrix}.
\]
The characteristic polynomial of $2M$ has the form
$$\det(2M-\lambda I)=\lambda^8-16\lambda^7+102\lambda^6-328\lambda^5+553\lambda^4-456\lambda^3+147\lambda^2-12\lambda$$
$$=\lambda(\lambda-4)(\lambda^3-6\lambda^2+9\lambda-3)(\lambda^3-6\lambda^2+9\lambda-1)$$
$$=\lambda(\lambda-4)\prod_{k\in\{1,2,4,5,7,8\}}
\Big(\lambda-2-2\cos\frac{k\pi}{9}\Big).$$
Therefore, the spectrum of $2M$ is given by
$$\sigma(2M)=\Big\{4,2+2\cos{\frac{\pi}{9}},2+2\cos{\frac{2\pi}{9}},2+2\cos{\frac{4\pi}{9}},2+2\cos{\frac{5\pi}{9}}, 2+2\cos{\frac{7\pi}{9}},2+2\cos{\frac{8\pi}{9}} ,0\Big\}.$$
All eigenvalues of $2M$ are non-negative, hence $M$ is positive semidefinite, and by Theorem \ref{tree}, the graph $\Theta(2,3,5)$ is of QE class.

To conclude this part, consider the graph $\Theta(2,3,7)$ and its oriented spanning tree, see Figure \ref{237}.

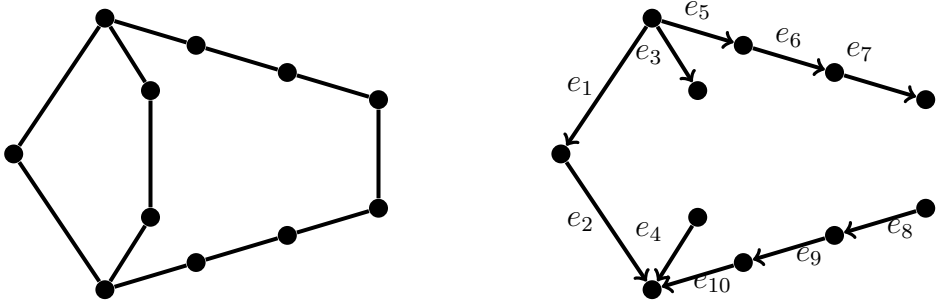
\begin{figure}[H]
\centering
\begin{tikzpicture}[x=12cm,y=12cm]
\tikzset{     
    e4c node/.style={circle,fill=black,minimum size=0.25cm,inner sep=0}, 
    e4c edge/.style={line width=1.5pt}
  }
  \node[e4c node] (1) at (0.15, 0.85) {}; 
  \node[e4c node] (2) at (0.05, 0.70) {}; 
  \node[e4c node] (3) at (0.15, 0.55) {}; 
  \node[e4c node] (4) at (0.20, 0.77) {}; 
  \node[e4c node] (5) at (0.20, 0.63) {}; 
  \node[e4c node] (6) at (0.25, 0.82) {};
  \node[e4c node] (7) at (0.35, 0.79) {};
  \node[e4c node] (21) at (0.45, 0.76) {};
  \node[e4c node] (22) at (0.45, 0.64) {};
  \node[e4c node] (8) at (0.35, 0.61) {};
 \node[e4c node] (9) at (0.25, 0.58) {};

  \node[e4c node] (11) at (0.75, 0.85) {}; 
  \node[e4c node] (12) at (0.65, 0.70) {}; 
  \node[e4c node] (13) at (0.75, 0.55) {}; 
  \node[e4c node] (14) at (0.80, 0.77) {}; 
  \node[e4c node] (15) at (0.80, 0.63) {}; 
  \node[e4c node] (16) at (0.85, 0.82) {};
  \node[e4c node] (17) at (0.95, 0.79) {};
  \node[e4c node] (121) at (1.05, 0.76) {};
  \node[e4c node] (122) at (1.05, 0.64) {};
  \node[e4c node] (18) at (0.95, 0.61) {};
 \node[e4c node] (19) at (0.85, 0.58) {};

  \path[draw,thick]
  (1) edge[e4c edge]  (2)
  (1) edge[e4c edge]  (4)
  (2) edge[e4c edge]  (3)
  (4) edge[e4c edge]  (5)
  (5) edge[e4c edge]  (3)
  (1) edge[e4c edge]  (6)
  (6) edge[e4c edge]  (7)
  %(7) edge[e4c edge]  (8)
  (8) edge[e4c edge]  (9)
  (9) edge[e4c edge]  (3)
  (7) edge[e4c edge]  (21)
  (21) edge[e4c edge]  (22)
  (22) edge[e4c edge]  (8)
  ;

   \path[->,draw,thick]
  (11) edge[e4c edge] node [anchor= east] {$e_1$} (12)
  (11) edge[e4c edge] node [anchor= east] {$e_3$} (14)
  (12) edge[e4c edge] node [anchor= east] {$e_2$} (13)
  %(14) edge[e4c edge]  (15)
  (15) edge[e4c edge] node [anchor= south east] {$e_4$} (13)
  (11) edge[e4c edge] node [anchor= south] {$e_5$} (16)
  (16) edge[e4c edge] node [anchor= south] {$e_6$} (17)
  %(7) edge[e4c edge]  (8)
  (18) edge[e4c edge]  (19)  node [anchor= north east] {$e_9$}
  (19) edge[e4c edge]  (13) node [anchor= north east] {$e_{10}$}
  (17) edge[e4c edge]  (121) node [anchor= south west] {$e_7$}
  %(121) edge[e4c edge]  (122)
  (122) edge[e4c edge]  (18) node [anchor= north east] {$e_8$} (17)
  ;
  
\end{tikzpicture}

\caption{The graph $\Theta(2,3,7)$ and its oriented spanning tree} \label{237}
\end{figure}

Now, the matrix $2M=[2\langle e_i,e_j\rangle]_{i,j=1,\ldots,10}$ of the graph $\Theta(2,3,7)$ has the form
\[2M=\begin{bmatrix}
2 & 0 & 0 & 1 & 0 & 0 & 0 & 1 & 0 & 0 \\
0 & 2 & 1 & 0 & 0 & 0 & 1 & 0 & 0 & 0\\
0 & 1 & 2 & -1 & 0 & 0 & 0 & 1 & 0 & 0 \\
1 & 0 & -1 & 2 & 0 & 0 & 1 & 0 & 0 & 0\\
0 & 0 & 0 & 0 & 2 & 0 & 0 & -1 & -1 & 0\\
0 & 0 & 0 & 0 & 0 & 2 & 0 & 0 & -1 & -1\\
0 & 1 & 0 & 1 & 0 & 0 & 2 & 0 & 0 & -1\\
1 & 0 & 1 & 0 & -1 & 0 & 0 & 2 & 0 & 0\\
0 & 0 & 0 & 0 & -1 & -1 & 0 & 0 & 2 & 0\\
0 & 0 & 0 & 0 & 0 & -1 & -1 & 0 & 0 & 2
\end{bmatrix}.
\]
The characteristic polynomial of $2M$ has the form
$$\det(2M-\lambda I)=\lambda^{10}-20\lambda^9+168\lambda^8-768\lambda^7+2068\lambda^6-3312\lambda^5+3023\lambda^4-1400\lambda^3+240\lambda^2$$
$$=\lambda^2(\lambda-4)^2(\lambda-3)(\lambda-1)\bigg(\lambda-\frac{5+\sqrt{5}}{2}\bigg)\bigg(\lambda-\frac{5-\sqrt{5}}{2}\bigg)\bigg(\lambda-\frac{3+\sqrt{5}}{2}\bigg)\bigg(\lambda-\frac{3-\sqrt{5}}{2}\bigg).$$
Therefore, the spectrum of $2M$ is given by
$$\sigma(2M)=\bigg\{4,4,\frac{5+\sqrt{5}}{2},3,\frac{3+\sqrt{5}}{2},\frac{5-\sqrt{5}}{2},1,\frac{3-\sqrt{5}}{2},0,0\bigg\}.$$
All eigenvalues of $2M$ are non-negative, hence $M$ is positive semidefinite, and by Theorem \ref{tree}, the graph $\Theta(2,3,7)$ is of QE class.

(ii) We prove the non-embeddability of graphs $\Theta(2,3,\gamma)$ for odd $\gamma\ge9$. Consider the graph $\Theta(2,3,2k+7)$, with $k\ge1$, see Figure \ref{232k+7}.

\begin{figure}[h]
\centering
\begin{tikzpicture}[x=10cm,y=8cm]
\tikzset{     
    e4c node/.style={circle,fill=black,minimum size=0.25cm,inner sep=0}, 
    e4c edge/.style={line width=1.5pt}
  }
  \node[e4c node] (x0) at (0.15, 0.40) {}; 
  \node[anchor= north] at (0.15, 0.39) {$x_0$};
  \node[e4c node] (x1) at (0.00, 0.70) {}; 
  \node[anchor=east] at (-0.01, 0.70) {$x_1$};
  \node[e4c node] (x2) at (0.15, 1.00) {};
  \node[anchor=south] at (0.15, 1.01) {$x_2$};
  \node[e4c node] (y1) at (0.15, 0.60) {};
  \node[anchor=west] at (0.16, 0.60) {$y_1$};
  \node[e4c node] (y2) at (0.15, 0.80) {};
  \node[anchor=west] at (0.16, 0.80) {$y_2$};
  \node[e4c node] (z1) at (0.35, 0.40) {};
  \node[anchor=north] at (0.35, 0.39) {$z_1$};
  \node[e4c node] (z2k) at (0.35, 1.00) {};
  \node[anchor=south] at (0.35, 1.01) {$z_{2k+6}$};
  \node[e4c node] (zk) at (0.70, 0.40) {};
  \node[anchor=north] at (0.70, 0.39) {$z_{k+1}$};
  \node[e4c node] (5) at (0.70, 1.00) {};
  \node[anchor=south] at (0.70, 1.01) {$z_{k+6}$};
  \node[e4c node] (4) at (0.90, 1.00) {};
  \node[anchor=south] at (0.90, 1.01) {$z_{k+5}$};
  \node[e4c node] (a) at (0.90, 0.40) {};
  \node[anchor=north] at (0.90, 0.39) {$z_{k+2}$};
  \node[e4c node] (b) at (1.00, 0.60) {};
  \node[anchor=west] at (1.01, 0.60) {$z_{k+3}$};
  \node[e4c node] (c) at (1.00, 0.80) {}; 
  \node[anchor=west] at (1.01, 0.80) {$z_{k+4}$};

  \path[draw,thick]
  (z1) edge[e4c edge, dotted](zk)
  (x0) edge[e4c edge]  (z1)
  (x1) edge[e4c edge]  (x2)
  (x0) edge[e4c edge]  (x1)
  (y2) edge[e4c edge]  (x2)
  (y1) edge[e4c edge]  (y2)
  (x0) edge[e4c edge]  (y1)
  (zk) edge[e4c edge]  (a)
  (a) edge[e4c edge]  (b)
  (b) edge[e4c edge]  (c)
  (c) edge[e4c edge]  (4)
  (4) edge[e4c edge]  (5)
  (5) edge[e4c edge, dotted]  (z2k)
  (z2k) edge[e4c edge]  (x2)
  ;
\end{tikzpicture}

\caption{A graph $\Theta(2,3,2k+7)$} 
\label{232k+7}
\end{figure}
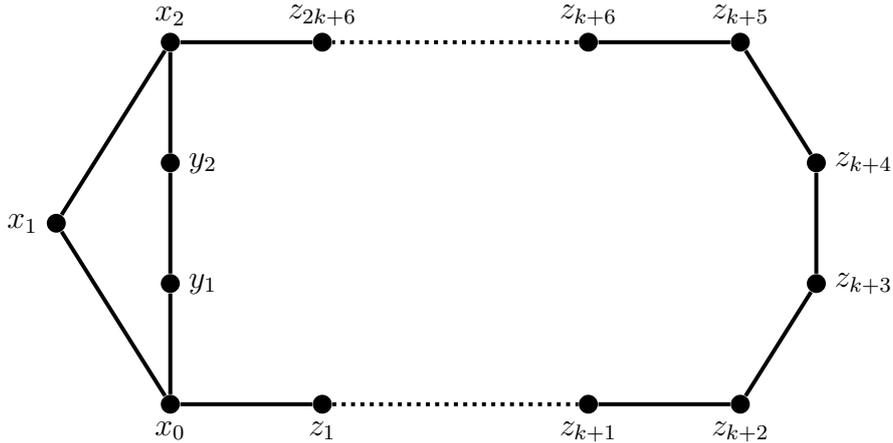

Put
\[
f(x_0)=f(x_2)=243,\quad
f(x_1)=-546,\quad
f(y_1)=f(y_2)=-206,
\]
\[
f(z_1)=f(z_{2k+6})=236,\quad
f(z_{k+1})=f(z_{k+6})=119,
\]
\[
f(z_{k+2})=f(z_{k+5})=234,\quad
f(z_{k+3})=f(z_{k+4})=-353
\]
and $f(x)=0$ for all other vertices. Let $\tilde{D}$ be the distance submatrix of the graph $\Theta(2,3,2k+7)$ with rows and columns indexed by $$z_{2k+6},x_0,x_1,x_2,y_1,y_2,z_1,z_{k+1},\ldots,z_{k+6}.$$

Then
$$\sum_{x,y\in V}d(x,y)f(x)f(y)=\langle g,\tilde{D}g\rangle=16272,$$
where

$$g=\begin{bmatrix}
236 &
    243 &
    -546 &
    243 &
    -206 &
    -206 &
    236 &
    119 &
    234 &
    -353 &
    -353 &
    234&
    119
\end{bmatrix}^\top$$
and

$$\tilde{D}=A+k\begin{bmatrix}
\mathbf{0}_{7\times 7} & \mathbf{1}_{7\times 6}\\[2pt]
\mathbf{1}_{6\times 7} & \mathbf{0}_{6\times 6}
\end{bmatrix},$$
where
\[
A=\begin{bmatrix}
0 & 3 & 2 & 1 & 3 & 2 & 4 & 4 & 4 & 3 & 2 & 1 & 0\\
3 & 0 & 1 & 2 & 1 & 2 & 1 & 1 & 2 & 3 & 4 & 4 & 3 \\
2 & 1 & 0 & 1 & 2 & 2 & 2 & 2 & 3 & 4 & 4 & 3 & 2 \\
1 & 2 & 1 & 0 & 2 & 1 & 3 & 3 & 4 & 4 & 3 & 2 & 1 \\
3 & 1 & 2 & 2 & 0 & 1 & 2 & 2 & 3 & 4 & 5 & 4 & 3 \\
2 & 2 & 2 & 1 & 1 & 0 & 3 & 3 & 4 & 5 & 4 & 3 & 2 \\
4 & 1 & 2 & 3 & 2 & 3 & 0 & 0 & 1 & 2 & 3 & 4 & 4  \\
4 & 1 & 2 & 3 & 2 & 3 & 0 & 0 & 1 & 2 & 3 & 4 & 5 \\
4 & 2 & 3 & 4 & 3 & 4 & 1 & 1 & 0 & 1 & 2 & 3 & 4 \\
3 & 3 & 4 & 4 & 4 & 5 & 2 & 2 & 1 & 0 & 1 & 2 & 3 \\
2 & 4 & 4 & 3 & 5 & 4 & 3 & 3 & 2 & 1 & 0 & 1 & 2 \\
1 & 4 & 3 & 2 & 4 & 3 & 4 & 4 & 3 & 2 & 1 & 0 & 1 \\
0 & 3 & 2 & 1 & 3 & 2 & 4 & 5 & 4 & 3 & 2 & 1 & 0 \\
\end{bmatrix}.
\]

Thus, the distance matrices of graphs $\Theta(2,3,2k+7)$ for $k\ge1$ are not conditionally negative definite, hence these graphs are of non-QE class. This concludes the proof.
\end{proof}

\section*{Acknowledgements}

The author would like to thank Wojciech Młotkowski for helpful discussions.

\bibliographystyle{plain}
\bibliography{bibliography}

\end{document}